\def\benm{\begin{enumerate}}
\def\eenm{\end{enumerate}}

\def\bal{\begin{align}}
\def\eal{\end{align}}

\documentclass[11pt]{amsart}
\usepackage{amsmath, amsthm, amscd, amsfonts,amssymb, graphicx}
\usepackage{color}

\newtheorem{theorem}{Theorem}[section]

\theoremstyle{definition}

\newtheorem{proposition}[theorem]{Proposition}
\newtheorem{corollary}[theorem]{Corollary}

\theoremstyle{remark}
\newtheorem{remark}[theorem]{Remark}

\numberwithin{equation}{section}

\newcommand{\dd}{\mathrm{d}}

\begin{document}

\title[Covariant Functions of Characters of Compact Subgroups]{Covariant Functions of Characters of Compact Subgroups}

\author[A. Ghaani Farashahi]{Arash Ghaani Farashahi}
\address{Department Pure Mathematics, School of Mathematics, University of Leeds, Leeds LS2 9JT, United Kingdom}
\email{a.ghaanifarashahi@leeds.ac.uk}
\email{ghaanifarashahi@outlook.com}

\curraddr{}







\thanks{E-mail addresses: a.ghaanifarashahi@leeds.ac.uk (Arash Ghaani Farashahi) }

\begin{abstract}
This paper presents a systematic study for abstract harmonic analysis on classical Banach spaces of covariant functions of characters of compact subgroups. Let $G$ be a locally compact group and $H$ be a compact subgroup of $G$. Suppose that $\xi:H\to\mathbb{T}$ is a continuous character, $1\le p<\infty$ and $L_\xi^p(G,H)$ is the set of all covariant functions of $\xi$ in $L^p(G)$.  It is shown that $L^p_\xi(G,H)$ is isometrically isomorphic to a quotient space of $L^p(G)$. It is also proven that $L^q_\xi(G,H)$ is isometrically isomorphic to the dual space $L^p_\xi(G,H)^*$, where $q$ is the conjugate exponent of $p$.
The paper is concluded by some results for the case that $G$ is compact.  
\end{abstract}

\maketitle

\section{\bf{Introduction}}

The Banach spaces consist of covariant functions on locally compact groups associated to continuous characters (one-dimensional continuous irreducible unitary representations) of closed subgroups appear in variant mathematical areas and their applications including calculus of pseudo-differential operators, number theory (automorphic forms), induced representations, 
homogenuos spaces, complex (hypercomplex) analysis, theoretical aspects of mathematical physics including coherent states and covariant analysis, see \cite{RB, FollP, kan.taylor, kisil.BJMA, kisil.book, kisil.A, kisil1, MackII, MackI}. 

The following paper presents a unified operator theoretic approach to study abstract harmonic analysis of $L^p$-spaces of covariant functions of characters of compact subgroups in locally compact groups. The introduced approach canonically extends classical methods of abstract harmonic analysis and functional analysis on coset spaces (homogenouse spaces) of compact subgroups by assuming the character to be the identity character of the compact subgroup, see \cite{der1983, For.RMJ, AGHF.CJM, Par.Kum}. 

This article contains 4 sections and organized as follows. Section  2  is  devoted  to  fix  notations  and  provides  a  summary
of classical harmonic analysis on locally compact compact groups and covariant functions of characters of closed subgroups.  
Let $G$ be a locally compact group, $H$ be a compact subgroup of $G$, and $1\le p<\infty$. Suppose that $\xi:H\to\mathbb{T}$ is a fixed character of $H$. 
In section 3, we study some of the operator theoretic aspects of covariant functions on $G$ associated to the character $\xi$ of the subgroup $H$.  
Next section investigates harmonic analysis foundations on the Banach space $L^p_\xi(G,H)$, the space of covariant functions of the character $\xi$ in $L^p(G)$. In this direction, we study fundamental properties of classical Banach spaces of covariant functions of characters of compact subgroups. It is shown that $L^p_\xi(G,H)$ is isometrically isomorphic to a quotient space of $L^p(G)$. We then proved that $L^q_\xi(G,H)$ is isometrically isomorphic to the dual space $L^p_\xi(G,H)^*$, where $q$ is the conjugate exponent of $p$. Finally, we conclude the paper by some related results for the case that $G$ is compact. 

\newpage
\section{\bf{Preliminaries and Notations}}

Let $X$ be a locally compact Hausdorff space. Then $\mathcal{C}_c(X)$ denotes the space of all continuous complex valued functions on $X$ with compact support. If $\lambda$ is a positive Radon measure on $X$, for each $1\le p<\infty$ the Banach space of equivalence classes of $\lambda$-measurable complex valued functions $f:X\to\mathbb{C}$ such that
$$\|f\|_{L^p(X,\lambda)}:=\left(\int_X|f(x)|^p\dd\lambda(x)\right)^{1/p}<\infty,$$
is denoted by $L^p(X,\lambda)$ which contains $\mathcal{C}_c(X)$ as a $\|.\|_{L^p(X,\lambda)}$-dense subspace.

Let $G$ be a locally compact group with the modular function $\Delta_G$ and a fixed left Haar measure $\lambda_{G}$. For a function $f:G\to\mathbb{C}$ and $x\in G$, the functions $L_xf,R_xf:G\to\mathbb{C}$ are given by $L_xf(y):=f(x^{-1}y)$ and $R_xf(y):=f(yx)$ for $y\in G$.
For $1\le p<\infty$, $L^p(G)$ stands for the Banach space $L^p(G,\lambda_G)$. The convolution for $f,g\in L^1(G)$, is defined via
\begin{equation}\label{f.ast.g}
f\ast_G g(x):=\int_Gf(y)g(y^{-1}x)\dd\lambda_G(y)\quad (x\in G).
\end{equation}
We then have $f\ast_G g\in L^p(G)$ with 
$\|f\ast_G g\|_{L^p(G)}\le \|f\|_{L^1(G)}\|g\|_{L^p(G)}$, if $f\in L^1(G)$ and $g\in L^p(G)$.
It is well known as a classical result in abstract harmonic analysis that the Banach function space $L^1(G,\lambda_G)$ is a Banach $*$-algebra with respect to the bilinear product $\ast_G:L^1(G)\times L^1(G)\to L^1(G)$ given by $(f,g)\mapsto f\ast_G g$, with $f\ast_G g$ defined by (\ref{f.ast.g}) and involution given by $f\mapsto f^{\ast_G}$ where $f^{*^G}(x):=\Delta_G(x^{-1})\overline{f(x^{-1})}$ for $x\in G$. Also, for each $p>1$, the Banach function space $L^p(G)$ is a Banach left $L^1(G)$-module equipped with  
the left module action $\ast_G:L^1(G)\times L^p(G)\to L^p(G)$ 
given by $(f,g)\mapsto f\ast_G g$, with $f\ast g$ is defined by (\ref{f.ast.g}), see \cite{der.main, FollH, HR1, kan.lau, 50} and the classical list of references therein. 

Suppose that $H$ is a closed subgroup of $G$. A character $\xi$ of $H$, is a continuous group homomorphism $\xi:H\to\mathbb{T}$, where $\mathbb{T}:=\{z\in\mathbb{C}:|z|=1\}$ is the circle group. In terms of group representation theory, each character of $H$ is a 1-dimensional irreducible continuous unitary representation of $H$. We then denote the set of all characters of $H$ by $\chi(H)$.

Let $G$ be a locally compact group, $H$ be a closed subgroup of $G$, and $\xi\in\chi(H)$. A function $\psi:G\to\mathbb{C}$ satisfies covariant property associated to the character $\xi$, if 
\begin{equation}\label{cov.prop.main}
\psi(xs)=\xi(s)\psi(x),
\end{equation}
for every $x\in G$ and $s\in H$.
The covariant functions appear in abstract harmonic analysis in the construction of induced representations, see \cite{FollH, kan.taylor}. We here employ some of the classical tools in this direction. Suppose that $\lambda_H$ is a left Haar measure on $H$.
For each character $\xi\in\chi(H)$ and a function 
$f\in\mathcal{C}_c(G)$, define the function $T_\xi(f):G\to\mathbb{C}$ via 
\[
T_\xi(f)(x):=\int_Hf(xs)\overline{\xi(s)}\dd\lambda_H(s),
\]
for every $x\in G$. 
Suppose $M_\xi(G,H)$ is the linear subspace of $\mathcal{C}(G)$ given by 
\[
M_\xi(G,H):=\{\psi\in\mathcal{C}_c(G|H):\psi(xh)=\xi(h)\psi(x),\ {\rm for\ all}\ x\in G,\ h\in H\},
\]
where 
\[
\mathcal{C}_c(G|H):=\{\psi\in\mathcal{C}(G):\mathrm{q}({\rm supp}(\psi))\ {\rm is\ compact\ in}\ G/H\},
\]
and $\mathrm{q}:G\to G/H$ is the canonical map given by $\mathrm{q}(x):=xH$ for $x\in G$. Using continuity of the canonical map $\mathrm{q}:G\to G/H$, one can deduce that $\mathcal{C}_c(G)\subseteq\mathcal{C}_c(G|H)$. It is shown that the linear operator $T_\xi$ maps $\mathcal{C}_c(G)$ onto $M_\xi(G,H)$, see Proposition 6.1 of \cite{FollH}.

\section{\bf Covariant Functions of Characters of Compact  Subgroups}
In this section, we shall study some of the fundamental theoretical aspects of covariant functions of characters of compact subgroup in locally compact groups. Throughout, let $G$ be a locally compact group, $H$ be a compact subgroup of $G$, and $\xi\in\chi(H)$ be a fixed character. 

\begin{proposition}\label{Jxi.Ncompact.Proj}
{\it Let $G$ be a locally compact group and $H$ be a compact subgroup of $G$. 
Suppose that $\xi\in\chi(H)$ is a character and $\lambda_H$ is the probability Haar measure of $H$. Then, 
\begin{enumerate}
\item $M_\xi(G,H)\subseteq\mathcal{C}_c(G)$.
\item $T_\xi\circ T_\xi=T_\xi$ on $\mathcal{C}_c(G)$.
\end{enumerate}
}\end{proposition}
\begin{proof}
(1)  Let $H$ be a compact subgroup of $G$. We then have $\mathcal{C}_c(G)=\mathcal{C}_c(G|H)$. To see this, let $\psi\in\mathcal{C}_c(G|H)$ be given. Then $\mathrm{q}({\rm supp}(\psi))$ is compact in $G/H$. Using Lemma 2.46 of \cite{FollH}, there exists a compact subset $F$ of $G$ such that $\mathrm{q}(F)=\mathrm{q}({\rm supp}(\psi))$. This implies that $\mathrm{supp}(\psi)\subseteq FH$. Since $H$ is compact, we conclude that $\mathrm{supp}(\psi)$ is compact in $G$ and so $\psi\in\mathcal{C}_c(G)$. Therefore, we deduce that $M_\xi(G,H)\subseteq\mathcal{C}_c(G)$.\\
(2) Let $f\in \mathcal{C}_c(G)$ be given. Using (1), we have $T_\xi(f)\in\mathcal{C}_c(G)$. Then, for each $x\in G$ and since $\lambda_H$ is a probability measure, we get
\begin{align*}
T_\xi(T_\xi(f))(x)&=\int_HT_\xi(f)(xs)\overline{\xi(s)}\dd\lambda_H(s)
\\&=\int_HT_\xi(f)(x)\xi(s)\overline{\xi(s)}\dd\lambda_H(s)
\\&=T_\xi(f)(x)\int_H|\xi(s)|^2\dd\lambda_H(s)
=T_\xi(f)(x)\left(\int_H\dd\lambda_H(s)\right)=T_\xi(f)(x).
\end{align*} 
\end{proof}

Invoking Proposition \ref{Jxi.Ncompact.Proj}, if $H$ is compact, one can regard the linear map $T_\xi:\mathcal{C}_c(G)\to\mathcal{C}_c(G)$ as a projection of the linear space $\mathcal{C}_c(G)$ onto the subspace $M_\xi(G,H)\subseteq\mathcal{C}_c(G)$. 

\begin{remark}
If $\xi=1$ is the trivial character of $H$, we then have $T_1(f)=T_H(f)$, see \cite{AGHF.IntJM, AGHF.IJM}.
Then, $M_1(G,H)$ consists of functions on $G$ which are constant on cosets of $N$. Therefore, $M_1(G,H)$ can be canonically identified with $\mathcal{C}_c(G/H)$ via the isometric identification $\psi\mapsto\widetilde{\psi}$, where $\widetilde{\psi}:G/H\to\mathbb{C}$ is given by $\widetilde{\psi}(xH):=\psi(x)$ for every $x\in G$, see Corollary 3.4 of \cite{AGHF.BMMSS}. In this case, harmonic analysis on $M_1(G,H)$ studied from different prespevtives in \cite{For.Sam.Sp, For.RMJ, AGHF.CJM, Par.Kum}.
\end{remark}

We then have the following observations.

\begin{proposition}\label{J.xi.Rk}
{\it Let $G$ be a locally compact group, $H$ be a compact subgroup of $G$, and $\xi\in\chi(H)$ be a character. Let $y\in G$ and $h\in H$. Then, 
\begin{enumerate}
\item $T_\xi\circ R_h=\xi(h)T_\xi$ on $\mathcal{C}_c(G)$.
\item $T_\xi\circ L_y=L_y\circ T_\xi$ on $\mathcal{C}_c(G)$.
\end{enumerate} 
}\end{proposition}
\begin{proof}
(1) Suppose $f\in\mathcal{C}_c(G)$ and $h\in H$. Let $x\in G$ be given. We then have  
\begin{align*}
T_\xi(R_hf)(x)&=\int_HR_hf(xs)\overline{\xi(s)}\dd\lambda_H(s)=\int_Hf(xsh)\overline{\xi(s)}\dd\lambda_H(s)
\\&=\int_Hf(xs)\overline{\xi(sh^{-1})}\dd\lambda_H(sh^{-1})
=\xi(h)\int_Hf(xs)\overline{\xi(s)}\dd\lambda_H(s)=\xi(h)T_\xi(f)(x),
\end{align*}
which implies that $T_\xi(R_hf)=\xi(h)T_\xi(f)$.\\
(2) Suppose $f\in\mathcal{C}_c(G)$ and $y\in G$. Let $x\in G$ be given. We then have
\begin{align*}
T_\xi(L_yf)(x)=\int_HL_yf(xs)\overline{\xi(s)}\dd\lambda_H(s)=\int_Hf(y^{-1}xs)\overline{\xi(s)}\dd\lambda_H(s)=L_y(T_\xi(f))(x),
\end{align*}
implying that $T_\xi(L_yf)=L_y(T_\xi(f))$.
\end{proof}

For functions $f,g\in\mathcal{C}_c(G)$, the function $f\overline{g}$ is  continuous with compact support. Therefore, it is integrable over $G$ with respect to every left Haar measure $\lambda_G$ on $G$. We denote the later integral by $\langle f,g\rangle$. 

\begin{theorem}
{\it Let $G$ be a locally compact group and $H$ be a compact subgroup of $G$. 
Suppose $\xi\in\chi(H)$ is a character and $f,g\in\mathcal{C}_c(G)$. We then have 
\begin{equation}\label{Jxi.Jxi*}
\langle T_\xi(f),g\rangle=\langle f,T_\xi(g)\rangle.
\end{equation}
}\end{theorem}
\begin{proof}
Let $f,g\in\mathcal{C}_c(G)$ be given. We then have 
\begin{align*}
\langle T_\xi(f),g\rangle&=\int_GT_\xi(f)(x)\overline{g(x)}\dd\lambda_G(x)
\\&=\int_G\left(\int_Hf(xs)\overline{\xi(s)}\dd\lambda_H(s)\right)\overline{g(x)}\dd\lambda_G(x)
\\&=\int_H\left(\int_Gf(xs)\overline{g(x)}\dd\lambda_G(x)\right)\overline{\xi(s)}\dd\lambda_H(s)
\\&=\int_H\left(\int_Gf(x)\overline{g(xs^{-1})}\dd\lambda_G(xs^{-1})\right)\overline{\xi(s)}\dd\lambda_H(s)
\\&=\int_H\Delta_G(s^{-1})\left(\int_Gf(x)\overline{g(xs^{-1})}\dd\lambda_G(x)\right)\overline{\xi(s)}\dd\lambda_H(s).
\end{align*}
Since $H$ is compact in $G$, we have $\Delta_G|_H=\Delta_H=1$. Therefore, we get  
\begin{align*}
\langle T_\xi(f),g\rangle
&=\int_H\Delta_H(s^{-1})\left(\int_Gf(x)\overline{g(xs^{-1})}\dd\lambda_G(x)\right)\overline{\xi(s)}\dd\lambda_H(s)
\\&=\int_Gf(x)\left(\int_H\overline{g(xs^{-1})}\overline{\xi(s)}\dd\lambda_H(s)\right)\dd\lambda_G(x)
\\&=\int_Gf(x)\left(\int_H\overline{g(xs)}\xi(s)\dd\lambda_H(s^{-1})\right)\dd\lambda_G(x)
\\&=\int_Gf(x)\left(\int_H\overline{g(xs)}\xi(s)\dd\lambda_H(s)\right)\dd\lambda_G(x)=\langle f,T_\xi(g)\rangle.
\end{align*}
\end{proof}

We then conclude this section by the following convolution property of the linear map $T_\xi$.

\begin{theorem}\label{mult.Cc}
Let $G$ be a locally compact group, $H$ be a compact subgroup of $G$, and $\xi\in\chi(H)$ be a character. Suppose $f,g\in\mathcal{C}_c(G)$ are given. We then have
\[
T_\xi(f\ast_G g)=f\ast_G T_\xi(g).
\]
\end{theorem}
\begin{proof}
Let $G$ be a locally compact group and $H$ be a compact subgroup of $G$.
Suppose that $f,g\in\mathcal{C}_c(G)$ and $x\in G$ are given. We then have  
\begin{align*}
T_\xi(f\ast_G g)(x)
&=\int_H f\ast_G g(xs)\overline{\xi(s)}\dd\lambda_H(s)
\\&=\int_H \left(\int _Gf(y)g(y^{-1}xs)\dd\lambda_G(y)\right)\overline{\xi(s)}\dd\lambda_H(s)
\\&=\int _Gf(y)\left(\int_H g(y^{-1}xs)\overline{\xi(s)}\dd\lambda_H(s)\right)\dd\lambda_G(y)
\\&=\int_Gf(y)T_\xi(g)(y^{-1}x)\dd\lambda_G(y)=f\ast_G T_\xi(g)(x).
\end{align*}
\end{proof}

\section{\bf $L^p$-Spaces of Covariant Functions of Characters of Compact  Subgroups}
In this section, we study $L^p$-spaces of covariant functions of characters of compact subgroups. We then investigate some of the basic properties of these classical Banach spaces of covariant functions of characters of compact  subgroups on locally compact groups. Throughout, suppose that $G$ is a locally compact group, $H$ is a compact subgroup of $G$, and $\xi\in\chi(H)$. Let $\lambda_{G}$ be a left Haar measure on $G$ and $\lambda_H$ be a Haar measure on $H$.

The following theorem proves the boundedness property of the linear map $T_\xi$ in the sense of $L^p(G)$.

\begin{theorem}\label{Js.p.Ncompact}
Let $G$ be a locally compact group and $1\le p<\infty$.
Suppose $H$ is a compact subgroup of $G$ and $\xi\in\chi(H)$.  The linear operator $T_\xi:(\mathcal{C}_c(G),\|.\|_{L^p(G)})\to(M_\xi(G,H),\|.\|_{L^p(G)})$ 
is bounded. In particular, if $\lambda_H$ is the probability Haar measure 
of $H$ then the linear operator  
$T_\xi:(\mathcal{C}_c(G),\|.\|_{L^p(G)})\to(M_\xi(G,H),\|.\|_{L^p(G)})$ 
is a contraction.
\end{theorem}
\begin{proof}
Let $f\in\mathcal{C}_c(G)$ be given. Since $H$ is compact and $\lambda_H(H)<\infty$, we have 
\begin{equation}\label{altJp}
\left(\int_H|f(ys)|\dd\lambda_H(s)\right)^p\le\lambda_H(H)^{p-1}\int_H|f(ys)|^p\dd\lambda_H(s),
\end{equation}
for each $y\in G$. Using compactness of $H$, we get $\Delta_G|_H=\Delta_H=1$. Then, compactness of $H$, implies that $\Delta_G(s)=1$ for all $s\in H$. Thus, we get   
\begin{align*}
\|T_\xi(f)\|_{L^p(G)}^p
&=\int_G\left|\int_Hf(xs)\overline{\xi(s)}\dd\lambda_H(s)\right|^p\dd\lambda_G(x)
\\&\le\int_G\left(\int_H|f(xs)|\dd\lambda_H(s)\right)^p\dd\lambda_G(x)
\\&\le\lambda_H(H)^{p-1}\int_G\int_H|f(xs)|^p\dd\lambda_H(s)\dd\lambda_G(x)
\\&=\lambda_H(H)^{p-1}\int_H\int_G|f(xs)|^p\dd\lambda_G(x)\dd\lambda_H(s)
\\&=\lambda_H(H)^{p-1}\int_H\int_G|f(x)|^p\dd\lambda_G(xs^{-1})\dd\lambda_H(s)
\\&=\lambda_H(H)^{p-1}\|f\|_{L^p(G)}^p\int_H\Delta_G(s^{-1})\dd\lambda_H(s)
=\lambda_H(H)^{p}\|f\|_{L^p(G)}^p,
\end{align*}
which guarantees that $\|T_\xi(f)\|_{L^p(G)}\le\lambda_H(H)\|f\|_{L^p(G)}$. Since $f\in\mathcal{C}_c(G)$ was arbitrary, we conclude that the linear map 
$T_\xi:(\mathcal{C}_c(G),\|.\|_{L^p(G)})\to(M_\xi(G,H),\|.\|_{L^p(G)})$ 
is bounded with $\|T_\xi\|\le\lambda_H(H)$. In particular, if $\lambda_H$ is the probability Haar measure of $H$ then the operator norm $\|T_\xi\|\le1$. So the linear map $T_\xi:(\mathcal{C}_c(G),\|.\|_{L^p(G)})\to(M_\xi(G,H),\|.\|_{L^p(G)})$ is a contraction.
\end{proof}

\begin{corollary}\label{Js.p.Gcompact}
{\it Let $G$ be a compact group and $1\le p<\infty$.
Suppose $H$ is a closed subgroup of $G$ and $\xi\in\chi(H)$.  The linear map 
$T_\xi:(\mathcal{C}_c(G),\|.\|_{L^p(G)})\to(M_\xi(G,H),\|.\|_{L^p(G)})$ 
is bounded. In particular, if $\lambda_H$ is the probability Haar measure 
of $H$ then the linear operator  
$T_\xi:(\mathcal{C}_c(G),\|.\|_{L^p(G)})\to(M_\xi(G,H),\|.\|_{L^p(G)})$ 
is a contraction.}
\end{corollary}

\begin{remark}
If $\xi=1$ is the trivial character then Theorem \ref{Js.p.Ncompact} coincides with Proposition 3.4 of \cite{AGHF.BMMSS}, if $M_1(G,H)$ is identified with $\mathcal{C}_c(G/H)$.
\end{remark}

From now on, we assume that $\lambda_H$ is the probability Haar measure on $H$ and hence the linear operator $T_\xi:(\mathcal{C}_c(G),\|.\|_{L^p(G)})\to(M_\xi(G,H),\|.\|_{L^p(G)})$ is a contraction.

We then conclude the following property of $\|.\|_{L^p(G)}$.

\begin{proposition}\label{Jxi.p.p.inf}
{\it Let $G$ be a locally compact group and $H$ be a compact subgroup of $G$. Suppose $1\le p<\infty$, $\xi\in\chi(H)$ and $\psi\in M_\xi(G,H)$. Then,  
\[
\|\psi\|_{L^p(G)}=\inf\left\{\|f\|_{L^p(G)}:f\in\mathcal{C}_c(G), \ T_\xi(f)=\psi\right\}.
\]
}\end{proposition}
\begin{proof}
Let $\psi\in M_\xi(G,H)$ be given. Suppose that $\mathcal{B}_\psi:=\{f\in\mathcal{C}_c(G): T_\xi(f)=\psi\}$ and also $\gamma_\psi:=\inf\left\{\|f\|_{L^p(G)}:f\in \mathcal{B}_\psi\right\}$. Using Theorem \ref{Js.p.Ncompact}, for each $f\in\mathcal{B}_\psi$, we have $\|\psi\|_{L^p(G)}\le\|f\|_{L^p(G)}$. Hence, we get $\|\psi\|_{L^p(G)}\le\gamma_\psi$. Now, we claim that $\|\psi\|_{L^p(G)}\ge\gamma_\psi$ as well. To this end, since $H$ is compact, using Proposition \ref{Jxi.Ncompact.Proj}, we have $\psi\in\mathcal{C}_c(G)$ and $T_\xi(\psi)=\psi$. Therefore, $\psi\in\mathcal{B}_\psi$ and hence we get $\|\psi\|_{L^1(G)}\ge\gamma_\psi$, which completes the proof.
\end{proof}

\begin{remark}
Let $\xi=1$ be the trivial character of $H$. Then Proposition \ref{Jxi.p.p.inf} is a consequence of Prospostion 3.4 and Corollary 3.4 of \cite{AGHF.BMMSS}, if  $M_1(G,H)$ is identified with $\mathcal{C}_c(G/H)$.
\end{remark}

Let $\mathcal{N}_\xi=\mathcal{N}_\xi(G,H)$ be the kernel of the linear map $T_\xi$ in $\mathcal{C}_c(G)$, that is the linear subspace given by 
\[
\mathcal{N}_\xi(G,H):=\{f\in\mathcal{C}_c(G):T_\xi(f)=0\}.
\]
Then, $\mathcal{N}_\xi(G,H)$ is a closed linear subspace of $(\mathcal{C}_c(G),\|.\|_{L^p(G)})$ for every $1\le p<\infty$. Using Proposition \ref{J.xi.Rk}(1), we also have  
\begin{equation}\label{Rf.xif}
\mathrm{span}\{R_hf-\xi(h)f:h\in H,\ f\in\mathcal{C}_c(G)\}\subseteq\mathcal{N}_\xi(G,H).
\end{equation}
For every $1\le p<\infty$, suppose that $\mathfrak{X}_\xi(G,H):=\mathcal{C}_c(G)/\mathcal{N}_\xi(G,H)$ is the quotient normed space of $\mathcal{N}_\xi(G,H)$ in $\mathcal{C}_c(G)$, that is 
\[
\mathcal{C}_c(G)/\mathcal{N}_\xi(G,H)=\left\{f+\mathcal{N}_\xi:f\in\mathcal{C}_c(G)\right\},
\]
with the quotient norm given by 
\begin{equation}\label{[p]}
\|f+\mathcal{N}_\xi\|_{[p]}:=\inf\left\{\|f+g\|_{L^p(G)}:g\in\mathcal{N}_\xi\right\}.
\end{equation}
Also, let $\mathfrak{X}_\xi^p(G,H)$ be the Banach completion of the normed linear space $\mathfrak{X}_\xi(G,H)$ with respect to the quotient norm 
$\|.\|_{[p]}$. We may denote $\mathfrak{X}_\xi(G,H)$ by $\mathfrak{X}_\xi$ and $\mathfrak{X}^p_\xi(G,H)$ by $\mathfrak{X}_\xi^p$ at times. 

We also denote by $L^p_\xi(G,H)$ the Banach completion of the normed linear space $M_\xi(G,H)$ with respect to $\|.\|_{L^p(G)}$. We shall use the completion norm by $\|.\|_{L^p(G)}$ or just $\|.\|_{p}$ as well. 

It is then clear that $L^p_\xi(G,H)\subseteq L^p(G)$. Also, one can see that 
\[
L^p_\xi(G,H)=\{\psi\in L^p(G):R_h\psi=\xi(h)\psi,\ {\rm for}\ h\in H\}
\]
Next we conclude the following characterisation of $L^p_\xi(G,H)$. 

\begin{theorem}\label{X1L1}
{\it Let $G$ be a locally compact group and $H$ be a compact subgroup of $G$. Suppose $1\le p<\infty$ and $\xi\in\chi(H)$. Then, 
 $\mathfrak{X}^p_\xi(G,H)$ is isometrically isomorphic to the Banach space $L^p_\xi(G,H)$.
}\end{theorem}
\begin{proof}
Invoking the structure of the spaces $\mathfrak{X}^p_\xi(G,H)$ and $L^p_\xi(G,H)$, it is enough to show that $(M_\xi(G,H),\|.\|_{L^p(G)})$ is isometrically isomorphic with the quotient normed space $(\mathfrak{X}_\xi(G,H),\|\cdot\|_{[p]})$. Since $T_\xi:\mathcal{C}_c(G)\to M_\xi(G,H)$ is a surjective linear map, we conclude that the linear space $M_\xi(G,H)$ is isomorphic with the quotient linear space $\mathfrak{X}_\xi(G,H)$ via the canonical linear map $U_\xi:\mathfrak{X}_\xi(G,H)\to M_\xi(G,H)$ given by $U_\xi(f+\mathcal{N}_\xi):=T_\xi(f)$ for every $f\in\mathcal{C}_c(G)$. Further, the isomorphism $U_\xi$ is not only algebraic, but also isometric, if the quotient linear space $\mathfrak{X}_\xi(G,H)$ is equipped with the classical quotient norm (\ref{[p]}). Using Proposition \ref{Jxi.p.p.inf}, for $f\in\mathcal{C}_c(G)$, we have 
\begin{align*}
\|U_\xi(f+\mathcal{N}_\xi)\|_{L^p(G)}
&=\|T_\xi(f)\|_{L^p(G)}
\\&=\inf\left\{\|h\|_{L^p(G)}:T_\xi(h)=T_\xi(f)\right\}
\\&=\inf\left\{\|h\|_{L^p(G)}:h-f\in\mathcal{N}_\xi\right\}
\\&=\inf\left\{\|f+g\|_{L^p(G)}:g\in\mathcal{N}_\xi\right\}=\|f+\mathcal{N}_\xi\|_{[p]}.
\end{align*}
\end{proof}

Invoking Theorem \ref{Js.p.Ncompact}, one can conclude that if $H$ is a compact subgroup of $G$ then the bounded linear map $T_\xi:(\mathcal{C}_c(G),\|.\|_{L^p(G)})\to(M_\xi(G,H),\|.\|_{L^p(G)})$ has a unique extension to a bounded linear map from $L^p(G)$ onto $L^p_\xi(G,H)$, which we still denote it by $T_\xi$. Then $T_\xi\circ R_h=\xi(h)T_\xi$ and $T_\xi\circ L_y=L_y\circ T_\xi$, for every $y\in G$ and $h\in H$. It is easy to see that the extended map $T_\xi:L^p(G)\to L^p_\xi(G,H)$ is given by $f\mapsto T_\xi(f)$, where  
\[
T_\xi(f)(x)=\int_Hf(xs)\overline{\xi(s)}\dd\lambda_H(s),\ {\rm for}\ x\in G.
\]
In particular, if $\lambda_H$ is the probability Haar measure 
of $H$, the extended linear operator $T_\xi:L^p(G)\to L^p_\xi(G,H)$ 
is a contraction. 

Invoking structure of $L^p_\xi(G,H)$ and since $L^p(G)$ is a Banach $L^1(G)$-module, we get that $L^p_\xi(G,H)$ is a Banach $L^1(G)$-submodule of $L^p(G)$ as well. In particular, we conclude that $L^1_\xi(G,H)$ is a closed left ideal in $L^1(G)$. 

We then prove the following the following multiplier property concerning convolution structure of $L^p_\xi(G,H)$ in terms of the linear operator $T_\xi$. 

\begin{proposition}
{\it Let $G$ be a locally compact group and $H$ be a compact subgroup of $G$. Suppose $\xi\in\chi(H)$ and $1\le p<\infty$. Then, $T_\xi:L^p(G)\to L^p(G)$ is a $L^1(G)$-multiplier.
}\end{proposition}
\begin{proof}
Let $f\in\mathcal{C}_c(G)$ and $g\in L^p(G)$. Suppose $(g_n)\subset\mathcal{C}_c(G)$ with $g=\lim_n g_n$ in $L^p(G)$. Using boundedness of the linear operator $T_\xi:L^p(G)\to L^p(G)$, continuity of the module action in each argument, and Theorem \ref{mult.Cc}, we get 
\begin{align*}
T_\xi(f\ast_G g)
&=T_\xi(\lim_nf\ast_Gg_n)
\\&=\lim_nT_\xi(f\ast_Gg_n)
\\&=\lim_nf\ast_GT_\xi(g_n)=f\ast_GT_\xi(g).
\end{align*}
Using a similar method, we get $T_\xi(f\ast_G g)=f\ast_GT_\xi(g)$, if $f\in L^1(G)$.
\end{proof}

For $1\le p<\infty$, let $\mathcal{N}_\xi^p(G,H)$ be the kernel of extension of the linear map $T_\xi$ in $L^p(G)$, that is the linear subspace given by 
\[
\mathcal{N}_\xi^p(G,H):=\left\{f\in L^p(G):T_\xi(f)=0\ {\rm in}\ L^p_\xi(G,H)\right\}.
\]

\begin{proposition}\label{Lp.main.Np}
{\it Let $G$ be a locally compact group and $H$ be a compact subgroup of $G$. Suppose $1\le p<\infty$ and $\xi\in\chi(H)$. Then, $\mathcal{N}_\xi^p(G,H)$ is the closure of $\mathcal{N}_\xi(G,H)$ in $L^p(G)$.
}\end{proposition}
\begin{proof}
Let $\mathcal{X}$ be the closure of $\mathcal{N}_\xi(G,H)$ in $L^p(G)$. Suppose $f\in\mathcal{X}$ is given. Then, we have $f\in L^p(G)$ and 
$\lim_nf_n=f$ for some sequence $(f_n)\subset\mathcal{N}_\xi(G,H)$. 
Continuity of the extended $T_\xi:L^p(G)\to L^p_\xi(G,H)$ implies that $T_\xi(f)=0$ in $L^p_\xi(G,H)$. Therefore, we get $f\in \mathcal{N}_\xi^p(G,H)$.
Since $f$ was arbitrary, we deduce that $\mathcal{X}\subseteq\mathcal{N}_\xi^p(G,H)$. Conversely, let $f\in\mathcal{N}^p(G,H)$ be given. Then, $f\in L^p(G)$ with $T_\xi(f)=0$ in $L^p_\xi(G,H)$.
Suppose that $\varepsilon>0$ is given. Pick $h\in\mathcal{C}_c(G)$ with $\|f-h\|_{L^p(G)}<\varepsilon/2$.  Let $\phi:=T_\xi(h)$. We then define the function $g:G\to\mathbb{C}$ by $g(x):=h(x)-\phi(x)$ for every $x\in G$. Then, we have $g\in\mathcal{C}_c(G)$ and $T_\xi(g)=0$. Indeed, for $x\in G$, we have 
\[
T_\xi(g)(x)=\phi(x)-\int_H\phi(xs)\overline{\xi(s)}d\lambda_H(s)=\phi(x)-\phi(x)=0.
\]
Hence, we get $g\in\mathcal{N}_\xi(G,H)$. Also, we have 
\begin{align*}
\|h-g\|_{L^p(G)}&=\|\phi\|_{L^p(G)}=\|T_\xi(h)\|_{L^p(G)}
\\&\le\|T_\xi(h-f)\|_{L^p(G)}+\|T_\xi(f)\|_{L^p(G)}\le\|h-f\|_{L^p(G)}<\frac{\varepsilon}{2}.
\end{align*}
Therefore, we achieve 
\[
\|f-g\|_{L^p(G)}\le\|f-h\|_{L^p(G)}+\|h-g\|_{L^p(G)}<\varepsilon,
\]
which implies that $f$ is in the closure of $\mathcal{N}_\xi(G,H)$ in $L^p(G)$, that is $f\in\mathcal{X}$. Since $f\in\mathcal{N}_\xi^p(G,H)$ was given, we conclude that $\mathcal{N}_\xi^p(G,H)\subseteq\mathcal{X}$ as well.\\
\end{proof}

We then have the following interesting characterization of the Banach space $L_\xi^p(G,H)$ as a quotient space of $L^p(G)$. 

\begin{theorem}\label{Lp.main}
{\it Let $G$ be a locally compact group and $H$ be a compact subgroup of $G$. Suppose $1\le p<\infty$ and $\xi\in\chi(H)$. The Banach space $L_\xi^p(G,H)$ is isometrically isomorphic to the quotient Banach space $L^p(G)/\mathcal{N}_\xi^p(G,H)$.
}\end{theorem}
\begin{proof}
Applying Propostion \ref{Lp.main.Np} and Lemma 3.4.4 of \cite{50} we deduce that $\mathfrak{X}_\xi^p(G,H)$
is canonically isometric isomorphic to the quotient Banach space $L^p(G)/\mathcal{N}_\xi^p(G,H)$. Then, Theorem \ref{X1L1} implies that the Banach space $L_\xi^p(G,H)$ is canonically isometric isomorphic to the quotient Banach space $L^p(G)/\mathcal{N}_\xi^p(G,H)$.\\
\end{proof}

We continue by some results concerning adjoint of the linear map $T_\xi:L^p(G)\to L^p(G)$, when $H$ is compact. 

\begin{proposition}
{\it Let $G$ be a locally compact group and $H$ be a compact subgroup of $G$. Suppose $\xi\in\chi(H)$, and $1<p,q<\infty$ with $p^{-1}+q^{-1}=1$. The adjoint of the bounded linear map $T_\xi:L^p(G)\to L^p(G)$ can be identified by $T_\xi:L^q(G)\to L^q(G)$.
}\end{proposition}
\begin{proof}
Let $g\mapsto \Lambda_g$ be the canonical isometric isomorphism identification of  $L^q(G)$ as $L^p(G)^*$, where the bounded linear functional $\Lambda_g:L^p(G)\to\mathbb{C}$ is given by $\Lambda_g(f):=\langle f,g\rangle$ for every $f\in L^p(G)$. 
Suppose that $f,g\in\mathcal{C}_c(G)$ are given. Invoking the abstract structure of the adjoint linear map $T_\xi^*:L^p(G)^*\to L^p(G)^*$, and using (\ref{Jxi.Jxi*}), we get 
\[
T_\xi^*(\Lambda_g)(f)=\Lambda_g(T_\xi(f))=\langle T_\xi(f),g\rangle=\langle f,T_\xi(g)\rangle=\Lambda_{T_\xi(g)}(f).
\]
Since $f\in\mathcal{C}_c(G)$ was arbitrary, and using density of $\mathcal{C}_c(G)$ in $L^p(G)$, we get $T_\xi^*(\Lambda_g)=\Lambda_{T_\xi(g)}$. Then density of $\mathcal{C}_c(G)$ in $L^q(G)$ implies that $T_\xi^*(\Lambda_g)=\Lambda_{T_\xi(g)}$ for every $g\in L^q(G)$. So, $T_\xi$ identifies $T_\xi^*$.
\end{proof}

\begin{corollary}
{\it Let $G$ be a locally compact group and $H$ be a compact subgroup of $G$. Suppose $\xi\in\chi(H)$, $1\le p<\infty$, and $\lambda_{H}$ be the probability Haar measure of $H$. Then, $T_\xi:L^p(G)\to L^p(G)$ is the projection onto $L^p_\xi(G,H)$. In particular, $T_\xi:L^2(G)\to L^2(G)$ is the orthogonal projection onto $L_\xi^2(G,H)$.
}\end{corollary}

Next we obtain the following characterization for the dual space $L^p_\xi(G,H)^*$, if $p>1$.

\begin{theorem}\label{Lpdual}
{\it Let $G$ be a locally compact group and $H$ be a compact subgroup of $G$. Suppose $\xi\in\chi(H)$, and $1<p,q<\infty$ with $p^{-1}+q^{-1}=1$. Then $L^q_\xi(G,H)$ is isometrically isomorphic to $L^p_\xi(G,H)^*$. 
}\end{theorem}
\begin{proof}
Let $g\mapsto \Lambda_g$ be the canonical isometric isomorphism identification of  $L^q(G)$ as $L^p(G)^*$, where $\Lambda_g:L^p(G)\to\mathbb{C}$ is given by $\Lambda_g(f):=\langle f,g\rangle$ for every $f\in L^p(G)$.
Invoking Theorem \ref{Lp.main}, we conclude that the dual space 
$L^p_\xi(G,H)^*$ is isometrically isomorphic to $\mathcal{N}_\xi^p(G,H)^\perp$, where 
\[
\mathcal{N}_\xi^p(G,H)^\perp=\left\{g\in L^q(G):\Lambda_g(f)=0,\ \ {\rm for\ all}\ f\in \mathcal{N}_\xi^p(G,H)\right\}.
\]
We then claim that $\mathcal{N}_\xi^p(G,H)^\perp=L^q_\xi(G,H)$. To show this, let $g\in\mathcal{N}_\xi^p(G,H)^\perp$ be given. Then, $g\in L^q(G)$ and 
$\Lambda_g(f)=0$ for all $f\in \mathcal{N}_\xi^p(G,H)$. Suppose that $h\in H$ is arbitrary. Then, using (\ref{Rf.xif}) and Theorem \ref{Lp.main.Np}, for every $f\in\mathcal{C}_c(G)$ we have $R_{h^{-1}}f-\overline{\xi(h)}f\in\mathcal{N}_\xi^p(G,H)$. This implies that $\Lambda_g(R_{h^{-1}}f)=\overline{\xi(h)}\Lambda_g(f)$. Therefore, using compactness of $H$, we get
\begin{align*}
\Lambda_{R_hg}(f)&=\int_Gf(x)\overline{g(xh)}\dd\lambda_G(x)
\\&=\int_Gf(xh^{-1})\overline{g(x)}\dd\lambda_G(xh^{-1})
\\&=\Delta_G(h^{-1})\int_Gf(xh^{-1})\overline{g(x)}\dd\lambda_G(x)
\\&=\Delta_H(h^{-1})\int_Gf(xh^{-1})\overline{g(x)}\dd\lambda_G(x)
\\&=\int_Gf(xh^{-1})\overline{g(x)}\dd\lambda_G(x)
\\&=\Lambda_g(R_{h^{-1}}f)=\overline{\xi(h)}\Lambda_g(f)=\Lambda_{\xi(h)g}(f).
\end{align*} 
Since $f\in\mathcal{C}_c(G)$ was arbitrary, we get $R_hg=\xi(h)g$ in $L^q(G)$. Since $h\in H$ was arbitrary, we conclude that $g\in L^q_\xi(G,H)$. Conversely, suppose that $g\in L^q_\xi(G,H)$. We shall show that $\Lambda_g(f)=0$ for every  
$f\in \mathcal{N}_\xi^p(G,H)$. According to Theorem 2.49 of \cite{FollH}, let $\lambda_{G/H}$ be the $G$-invariant Radon measure on the left coset space $G/H$ normalized with respect to the Weil's formula (2.50) of \cite{FollH}.
Then, for every  $f\in \mathcal{N}_\xi^p(G,H)$, we have 
\begin{align*}
\Lambda_g(f)&=\int_G f(x)\overline{g(x)}\dd\lambda_G(x)
\\&=\int_{G/H}\left(\int_Hf(xh)\overline{g(xh)}\dd\lambda_H(h)\right)\dd\lambda_{G/H}(xH)
\\&=\int_{G/H}\left(\int_Hf(xh)\overline{\xi(h)}\dd\lambda_H(h)\right)\overline{g(x})\dd\lambda_{G/H}(xH)
=\int_{G/H}T_\xi(f)(x)\overline{g(x})\dd\lambda_{G/H}(xH)=0,
\end{align*}
which implies that $g\in \mathcal{N}_\xi^p(G,H)^\perp$.
\end{proof}

Suppose that $L^\infty(G)$ is the Banach space of all locally $\lambda_G$-measurable functions $f:G\to\mathbb{C}$ that are bounded except on a locally $\lambda_G$-null set, modulo functions which are zero locally a.e. on $G$ given the norm
\[
\|f\|_\infty:=\inf\{t: |f(x)|\le t\ \ {\rm l.a.e.}\ \ x\in G\}.
\]
We then have $M_\xi(G,H)\subset L^\infty(G)$. It is also routine to check that the linear operator 
$T_\xi:(\mathcal{C}_c(G),\|\cdot\|_{L^\infty(G)})\to(M_\xi(G,H),\|\cdot\|_{L^\infty(G)})$, is bounded with the operator norm $\|T_\xi\|\le\lambda_H(H)$. In particular, if $\lambda_H$ is the probability Haar measure 
of $H$ then the linear operator  
$T_\xi:(\mathcal{C}_c(G),\|\cdot\|_{L^\infty(G)})\to(M_\xi(G,H),\|\cdot\|_{L^\infty(G)})$ 
is a contraction. 

Let $L^\infty_\xi(G,H)$ be  closed subspace of $L^\infty(G)$ given by  
\[
L^\infty_\xi(G,H):=\left\{\psi\in L^\infty(G):R_h\psi=\xi(h)\psi,\ {\rm for}\ h\in H\right\}.
\]

We then also obtain the following characterization for the dual space $L^1_\xi(G,H)^*$.

\begin{theorem}
{\it Let $G$ be a locally compact group and $H$ be a compact subgroup of $G$. Suppose $\xi\in\chi(H)$. Then $L^1_\xi(G,H)^*$ is isometrically isomorphic to $L^\infty_\xi(G,H)$. 
}\end{theorem}
\begin{proof}
Let $g\mapsto \Lambda_g$ be the canonical isometric isomorphism identification of $L^\infty(G)$ as $L^1(G)^*$, where $\Lambda_g:L^1(G)\to\mathbb{C}$ is given by $\Lambda_g(f):=\langle f,g\rangle$ for $f\in L^1(G)$, see Theorem 12.18 of \cite{HR1}.
Invoking Theorem \ref{Lp.main}, the dual space 
$L^1_\xi(G,H)^*$ is isometric isomorphic to $\mathcal{N}_\xi^1(G,H)^\perp$, where 
\[
\mathcal{N}_\xi^1(G,H)^\perp=\left\{g\in L^\infty(G):\Lambda_g(f)=0,\ \ {\rm for\ all}\ f\in \mathcal{N}_\xi^1(G,H)\right\}.
\]
Then using a similar method used in Theorem \ref{Lpdual}, we get $\mathcal{N}_\xi^1(G,H)^\perp=L^\infty_\xi(G,H)$. 
\end{proof}

We then conclude the paper by the following inclusion property when $G$ is compact. 

\begin{proposition}
{\it Let $G$ be a compact group and $H$ be a closed subgroup of $G$. 
Suppose $\xi\in\chi(H)$ and $1\le p<\infty$. Then, $L^p_\xi(G,H)\subseteq L_\xi^1(G,H)$.
}\end{proposition}
\begin{proof}
Since $G$ is compact, we have $\|\psi\|_{L^1(G)}\le\lambda_G(G)^{\frac{p-1}{p}}\|\psi\|_{L^p(G)}$, for $\psi\in\mathcal{C}_c(G)$. Therefore, $\|\psi\|_{L^1(G)}\le\lambda_G(G)^{\frac{p-1}{p}}\|\psi\|_{L^p(G)}$ for $\psi\in M_\xi(G,H)$. This implies that $L^p_\xi(G,H)\subseteq L^1_\xi(G,H)$.
\end{proof}

\begin{corollary}
{\it Let $G$ be a compact group and $H$ be a closed subgroup of $G$. 
Suppose $\xi\in\chi(H)$ and $1\le p<\infty$. Then, 
\begin{enumerate}
\item $T_\xi:L^p(G)\to L^p(G)$ is a $L^p(G)$-multiplier.
\item $L^p_\xi(G,H)$ is a closed left ideal in $L^p(G)$. 
\end{enumerate}
}\end{corollary}

\

{\bf Acknowledgement.}
This project has received funding from the European Union’s Horizon 2020 research and innovation programme under the Marie Sklodowska-Curie grant agreement No. 794305. The author gratefully acknowledges the supporting agency. The findings and opinions expressed here are only those of the author, and not of the funding agency.\\
The author would like to express his deepest gratitude to Vladimir V. Kisil for suggesting the problem that motivated the results in this article, 
stimulating discussions and pointing out various references. 

\bibliographystyle{amsplain}

\end{document}